\documentclass[10pt]{article}

\usepackage{amsmath}    % need for subequations
\usepackage{graphicx}   % need for figures
\usepackage{verbatim}   % useful for program listings
\usepackage{color}      % use if color is used in text
\usepackage{hyperref}   % use for hypertext links, including those to external documents and URLs

\usepackage{amssymb}
\usepackage{amsthm}
\usepackage{utopia}
\usepackage[top=2cm, bottom=2cm, left=2cm, right=2cm]{geometry}

\theoremstyle{theorem}
\newtheorem{Def}{Definition}[section]

\newtheorem{Lem}[Def]{Lemma}
\newtheorem{Thm}[Def]{Theorem}

\newtheorem{Ass}[Def]{Assumption}
\theoremstyle{definition}
\newtheorem{Rem}[Def]{Remark}

\newcommand{\p}{\mathbb{P}}
\newcommand{\e}{\mathbb{E}}
\newcommand{\real}{\mathbb{R}}
\newcommand{\n}{\mathbb{N}}

\newcommand{\1}{{\bf 1}}

\newcommand{\ksm}{K_\sigma}
\newcommand{\hbeta}{\hat{\beta}}
\begin{document}

\title{Strong convergence for the Euler-Maruyama approximation of stochastic differential equations with discontinuous coefficients}
\author
{
	Hoang-Long Ngo\footnote{Hanoi National University of Education, 136 Xuan 
	Thuy - Cau Giay - Hanoi - Vietnam, email: ngolong@hnue.edu.vn}
	$\quad $ and $\quad$ 
	Dai Taguchi\footnote{Ritsumeikan University, 1-1-1 Nojihigashi, Kusatsu, 
	Shiga, 525-8577, Japan, email: dai.taguchi.dai@gmail.com }
}

\date{}
\maketitle
\begin{abstract}
%In this paper we show the strong convergence at the rate of order $\frac{1}{\log n}$ for the Euler-Maruyama approximation of a class of stochastic differential equations whose both drift and diffusion coefficients are possibly discontinuous. \\ \\
In this paper we study the strong convergence for the Euler-Maruyama approximation of a class of stochastic differential equations whose both drift and diffusion coefficients are possibly discontinuous. 

%60H35 Computational methods for stochastic equations
%41A25 Rate of convergence, degree of approximation
%60H10 Stochastic ordinary differential equations
%65C30 Stochastic differential and integral equations
\end{abstract}

\section{Introduction} \label{Sec_1}
Let us consider the one-dimensional stochastic differential equation (SDE)
\begin{align}\label{SDE_1}
X_t=
x_0
+\int_0^t b(X_s)ds
+\int_0^t \sigma(X_s)dW_s,
~x_0 \in \real,
~t \in [0,T],
\end{align}
where $W:=(W_t)_{0\leq t \leq T}$ is a standard  Brownian motion defined on a probability space $(\Omega, \mathcal{F},\p)$ with a filtration $(\mathcal{F}_t)_{0\leq t \leq T}$ satisfying the usual conditions. Since the solution of (\ref{SDE_1}) is rarely analytically tractable, one often approximates $X=(X_t)_{0 \leq t \leq T}$ by using the Euler-Maruyama (EM) scheme given by 
\begin{align*}%\label{EM_1}
X_t^{(n)} 
&= x_0 +\int_0^tb\left(X_{\eta _n(s)}^{(n)}\right)ds +\int_0^t \sigma\left(X_{\eta _n(s)}^{(n)}\right) dW_s,~t \in [0,T],
\end{align*}
where $\eta _n(s) = kT/n=:t_k^{(n)}$ if $ s \in \left[kT/n, (k+1)T/n \right)$.

It is well-known that if $b$ and $\sigma$ are Lipschitz continuous, the EM approximation for \eqref{SDE_1} converges at the strong rate of order $1/2$ (see \cite{KP}). On the other hand, when $b$ and $\sigma$ are not Lipschitz continuous, the strong rate is less known and it has been  a subject of extensive study. In the recent articles  \cite{JMY} and \cite{HHJ}, it has been shown that for every arbitrarily slow convergence speed there exist SDEs with infinitely often differentiable and globally bounded coefficients such that neither the EM approximation nor any approximation method based on finitely many observations of the driving Brownian motion can converge in absolute mean to the solution faster than the given speed of convergence.
The  approximation for SDEs with possibly discontinuous drift coefficients was first studied in \cite{G98}. It is shown that if the drift satisfies the monotonicity condition and the diffusion coefficient is Lipschitz continuous, then the EM scheme converges at the rate of $1/4$ in pathwise senses. In \cite{HK}, the strong convergence of EM scheme is shown for SDEs with discontinuous monotone drift coefficients. 
If $\sigma$ is uniformly elliptic and  $(\alpha + 1/2)$-H\"oder continuous, and $b$ is of locally bounded variation, it has been shown that the strong rate of the EM in $L^1$-norm is $n^{-\alpha}$ for $\alpha \in (0,1/2]$ and $(\log n)^{-1}$ for $\alpha=0$  (see \cite{NT_MCOM, NT_2015_2}).
The strong rate of convergence for SDEs whose drift coefficient $b$ is H\"older continuous is studied in \cite{Gyongy, MeTa, NT_2015_2}. The above mentioned papers contain just a few selected results and a number of further and partially significantly improved approximation results for SDEs with irregular coefficients are available in the literature;
see, e.g., \cite{AKU, CS, HaTsu, HJK, KLY, LeSz, MT, NT_2015_1, Y} and the references there in.

In this paper we are interested in strong approximation of SDEs with discontinuous diffusion coefficients. These SDEs appears in many applied domains such as stochastic control and quantitative finance (see \cite{CE, AI}). 
For such SDEs, the existence and uniqueness of solution  was studied in \cite{Nakao, LeGall, CE}; the weak convergence of EM approximation  was shown in \cite{Y}.
To the best of our knowledge, the strong convergence of the EM approximation of SDEs with discontinuous diffusion coefficient  has not been considered before in the literature.
It is worth noting that the key ingredients to establish the strong rate of convergence of EM approximation for SDEs with discontinuous drift are either the Krylov estimate (see \cite{KLY, Gyongy}) or the Gaussian bound estimate for the density of the numerical solution (\cite{Lemaire, NT_MCOM, NT_2015_2}). However, these estimates seem no longer available for SDEs with discontinuous diffusion coefficients. Therefore in this paper we  develop another method, which is based on an argument with local time, to overcome this obstacle. 

The remainder of the paper is structured as follows. In the next section we introduce some notations and assumptions for our framework together with the main results. All proofs are deferred to Section 3.

\section{Main results}
\subsection{Notations}
Throughout this paper the following notations are used.
For any continuous semimartingale $Y$, we denote $L^x_t(Y)$ the symmetric local time of $Y$ up to time $t$ at the level $x \in \real$ (see \cite{LeGall}).
For bounded measurable function $f$ on $\real$, we define $\|f\|_{\infty}:=\sup_{x \in \real} |f(x)|$. We denote by $L^1(\real)$ the space of all integrable functions with respect to Lebesgue measure on $\real$ with semi-norm $\|f\|_{L^1(\real)}:=\int_{\real}|f(x)|dx$. For each $\beta \in (0,1]$ and $\kappa >0$, we denote by $H^{\beta, \kappa}$ the set of all functions $f:\real \to \real$ such that there exists  a measurable subset $S(f)$ of $\real$ satisfying 
\begin{itemize}
\item[(i)]  
$\|f\|_{\beta} := \|f\|_\infty+ \sup_{x<y; [x,y]\cap S(f) = \emptyset} \dfrac{|f(x)-f(y)|}{|x-y|^{\beta}} < \infty$; and
\item[(ii)] $C_{\beta,\kappa}:= \sup_{K\geq 1}\sup_{\varepsilon>0}  \dfrac{\lambda(S(f)^\varepsilon \cap [-K,K])}{K \varepsilon^\kappa} < +\infty$
where $\lambda$ denotes the Lebesgue measure on $\real$ and  
$S(f)^\varepsilon$ is the $\varepsilon$-neighbourhood of $S(f)$, i.e., $S(f)^\varepsilon = \{y \in \real: \text{ there exists } x \in S(f) \text{ such that } |x-y|\leq \varepsilon\}.$
\end{itemize}
%We also denote by $H^\beta$ the set of functions $f$ satisfying condition (i).
Here are some remarks on the class $H^{\beta, \kappa}$.
\begin{Rem}
\begin{enumerate}
%\item If $f$ is bounded and uniformly $\beta$-H\"older continuous on $\real$ then $f \in H^{\beta, \kappa}$ for any $\kappa>0$.
%\item It is easy to verify that if $f \in H^\beta$ and the cardinality of $S(f)$ is finite then $f \in H^{\beta, 1}$.  
%\item  Let $f(x) = (1 \wedge |x|^{1/2})\sin([|x|^{\kappa-1}])$ for some $\kappa \in (0,1]$. We can show that $S(f) = \{k^{1/(\kappa-1)}: k \in \mathbb{Z}\backslash \{0\}\}\cup\{0\}$ and  $f \in  H^{1/2, \kappa}$. 
\item $H^{\beta, \kappa}$ is a vector space on $\mathbb{R}$, i.e., if $a, b \in \real$ and $f, g \in H^{\beta, \kappa}$ then $af + bg \in H^{\beta, \kappa}$.
\item A bounded function $f$ is called piecewise $\beta$-H\"older if there exist a positive constant $L$ and a sequence $-\infty = s_0 < s_1 < s_2 < \ldots < s_m < s_{m+1}= \infty$ such that 
$|f(u) - f(v)| \leq L|u-v|^\beta$ for any $u,v$ satisfying $s_k < u < v < s_{k+1}$. It is easy to verify that such function $f \in H^{\beta,1}, \ S(f) = \{s_1, \ldots, s_m\}$ and $C_{\beta,  1} \leq 2m$.
\item The following  $\zeta$ is a non-trivial example of function of $H^{\beta, \kappa}$ with $\kappa < 1$.  
For each $\hbeta, \kappa \in (0,1)$, we denote 
\begin{equation}\label{exp:zeta} \zeta(x) = 
\begin{cases}  \frac{x-1}{2x-1} & \text{ if } x \leq 0,\\ 
1+\frac{\log 2}{\log(n+1)}x^{\hbeta} & \text{ if } (n+1)^{-1/(1-\kappa)} \leq x < n^{-1/(1-\kappa)} \text{ and } \ n \in \n, \\ 
\frac{3x+1}{x+1} &\text{ if } x \geq 1.\end{cases}
\end{equation}
 It can be shown that $\zeta$ is a strictly increasing function with an infinite number of discontinuous points which are cumulative at $0$, $\frac{1}{2} < \zeta < 3$, and $\zeta \in H^{\beta,\kappa}$ with $\beta = \frac{1+\hbeta-\kappa}{2-\kappa}$,  $S(\zeta) = \{ n^{-1/(1-\kappa)}, n = 1, 2,\ldots \}$ and $C_{\beta,\kappa} \leq 3$.
\end{enumerate}
\end{Rem}
%We denote by $c(p)$ the constant in Burkholder-Davis-Gundy's inequality for $p >0$.

\subsection{Main results}

We need the following assumptions on the diffusion coefficient $\sigma$.

\begin{Ass}\label{Ass_1}
%	We assume that the diffusion coefficient $\sigma$ satisfies the following conditions:
	\begin{itemize}
		\item[(i)] There exists a bounded and  strictly increasing function $f_{\sigma}$ such that for any $x,y \in \real$,
		\begin{align*}
		|\sigma(x)-\sigma(y)|^2 \leq |f_{\sigma}(x)-f_{\sigma}(y)|.
		\end{align*}
		
		\item[(ii)]  $\sigma$ is bounded and uniformly positive, i.e. there exist positive constants $\overline{\sigma}$ and $\underline{\sigma}$ such that for any $x \in \real$, 
		\begin{align*}
			\underline{\sigma} \leq \sigma(x) \leq \overline{\sigma}.
		\end{align*}
		
		\end{itemize}
\end{Ass}

Le Gall \cite{LeGall} has shown that if $b$ is bounded measurable, and $\sigma$ satisfies Assumption \ref{Ass_1}, then there exists a unique strong solution to SDE \eqref{SDE_1} (see also  \cite{Nakao}). We now give some remarks on the Assumption \ref{Ass_1}.
\begin{Rem}
\begin{enumerate}
%\item 
\item The function  $\sigma(x) = 1 + \1_{x \geq 0}$ satisfies Assumption \ref{Ass_1} and  belongs to $H^{1,1}$.
\item The function $\zeta$ defined in \eqref{exp:zeta} also satisfies Assumption \ref{Ass_1}.
\item If $a, b >0$ and $\sigma_1, \sigma_2$ satisfies Assumption \ref{Ass_1}, then $a\sigma_1 + b\sigma_2$ also satisfies Assumption \ref{Ass_1}. 
\item Let $f_1, f_2$ be two  strictly increasing, piecewise $1$-H\"older functions. Let $\rho$ be a $1/2$-H\"older continuous function satisfying $0 < \inf_{x \in \real} \rho(x) \leq \sup_{x \in \real} \rho(x) < \infty$.  Then $\sigma := \rho\circ (f_1-f_2)$ is piecewise $1/2$-H\"older and it satisfies Assumption \ref{Ass_1} with $f_\sigma =   C(f_1 + f_2)$ for some positive constant $C$. 
\end{enumerate}
\end{Rem}
We are now in the position to state the main result of this paper. 
\begin{Thm} \label{Main_1}
Let Assumption \ref{Ass_1} hold, and  $b, \sigma \in H^{\beta, \kappa}$ for some $\beta \in (0,1]$ and $\kappa >0$.
%, and satisfies Assumption \ref{Ass_1}. Suppose that $b \in  H^{\beta, \kappa} \cap  L^1(\real)$.
%Moreover, suppose that one of the following conditions on $b$ holds
%\begin{enumerate}
%\item[(A)] $ b \in H^{\beta, \kappa} \cap  L^1(\real)$;
%\item[(B)] $b$ is bounded and one-sided Lipschitz continuous, i.e., there exists a constant $L>0$ such that 
%$$(x-y)(b(x) - b(y)) \leq L (x-y)^2, \quad \text{ for all } x, y \in \real.$$
%\end{enumerate}
\begin{enumerate}
\item[(i)] There exists a constant $C$ such that for all $n \geq 3$,
\begin{equation} \label{logn2}
\sup_{0\leq t \leq T} \e[|X_t - X^{(n)}_t|] \leq \frac{Ce^{C\sqrt{\log \log n}}}{\log n}.
\end{equation}
\item[(ii)] Moreover, if $ b \in L^1(\real)$, then there exists a constant $C$ such that for all $n \geq 3$,
\begin{equation} \label{logn} 
\sup_{0\leq t \leq T} \e[|X_t - X^{(n)}_t|] \leq \frac{C}{\log n}.
\end{equation}
\end{enumerate}
\end{Thm}
%Gy\"ongy and R\'asonyi \cite{Gyongy} obtained 
The estimates \eqref{logn2} and \eqref{logn} were obtained in \cite{Gyongy, NT_MCOM, NT_2015_2} under a stronger assumption that $\sigma$ is $1/2$-H\"older continuous on $\mathbb{R}$. 
\section{Proof of main results}

\subsection{Some auxiliary estimates}

In this section, we derive a key estimation (Lemma \ref{key_sigma_0}) for proving the main theorem.
We first introduce the following standard estimation (see Remark 1.2 in \cite{Gyongy}).

\begin{Lem} \label{Lem_1}
	Suppose that $b$ and $\sigma$ are bounded, measurable.
	Then for any $q>0$, there exists $C_q \equiv C(q,\|b\|_{\infty}, \|\sigma\|_{\infty}, T) $ such that for all $n \in \n$,
	\begin{align*}
	\sup_{t \in [0,T]} \e[|X_t^{(n)}-X_{\eta_n(t)}^{(n)}|^q]\leq \frac{C_q}{n^{q/2}}.
	\end{align*}
\end{Lem}

The next estimation is a uniform $L^2$-bounded of the local time of solution of SDE \eqref{SDE_1} and its EM approximation.

\begin{Lem}\label{local_time}
	Suppose that $b$ is bounded, measurable and $\sigma$ is measurable and satisfies Assumption \ref{Ass_1}-(ii).
	For each $\theta \in [0,1]$, define
	\begin{align*}
	&V_t^{(n)}(\theta):=(1-\theta)X_t+\theta X_t^{(n)}.\\
	&=x_0
	+\int_{0}^{t} \left\{ (1-\theta)b(X_s) + \theta b(X_{\eta_n(s)}^{(n)}) \right\} ds
	+\int_{0}^{t} \left\{ (1-\theta)\sigma(X_s) + \theta \sigma(X_{\eta_n(s)}^{(n)}) \right\} dW_s.
	\end{align*}
	Then it holds that
	\begin{align}\label{esti_local_time_0}
	\sup_{\theta \in [0,1], x \in \real}
	\e[|L_T^x(V^{(n)}(\theta))|^2]
	&\leq 12\|b\|_{\infty}^2T^2+ 6 \overline{\sigma}^2 T.
	\end{align}
\end{Lem}
\begin{proof}
	By using the symmetric It\^o-Tanaka formula, we have
	\begin{align*}
	L_T^x(V^{(n)}(\theta))
	&=|V_T^{(n)}(\theta)-x|-|x_0-x|-\int_0^T \left( \1(V_s^{(n)}(\theta)>x)-\1(V_s^{(n)}(\theta)<x) \right) dV_s^{(n)}(\theta)\\
	&\leq |V_T^{(n)}(\theta)-x_0|+\left| \int_0^T \left( \1(V_s^{(n)}(\theta)>x)-\1(V_s^{(n)}(\theta)<x) \right) dV_s^{(n)}(\theta) \right|\\
	&\leq 2\int_{0}^{T} \left| (1-\theta)b(X_s) + \theta b(X_{\eta_n(s)}^{(n)}) \right| ds
	+\left| \int_{0}^{T} \left\{ (1-\theta)\sigma(X_s) + \theta \sigma(X_{\eta_n(s)}^{(n)}) \right\} dW_s\right|\\
	&+\left| \int_{0}^{T} \left( \1(V_s^{(n)}(\theta)>x)-\1(V_s^{(n)}(\theta)<x) \right) \left\{ (1-\theta)\sigma(X_s) + \theta \sigma(X_{\eta_n(s)}^{(n)}) \right\} dW_s\right|.
	\end{align*}
	Since $b$ and $\sigma$ are bounded, it follows from inequality $(a+b+c)^2\leq 3(a^2+b^2+c^2)$  and the $L^2$-isometry that,
	\begin{align*}
	\sup_{\theta \in [0,1], x \in \real}
	\e[|L_T^x(V^{(n)}(\theta))|^2] \notag
	&\leq 
	12\|b\|_{\infty}^2T^2
	+ 6 \sup_{\theta \in [0,1], x \in \real}\int_0^T \e \Big[ \big| (1 - \theta) \sigma(X_s) + \theta \sigma(X^{(n)}_{\eta(s)})\big|^2 \Big] ds\\
	&\leq 12\|b\|_{\infty}^2T^2 + 6 \overline{\sigma}^2T.
	\end{align*}
	This concludes the statement.
	%By Jensen's inequality and \eqref{esti_local_time_3}, we have
	%\begin{align}\label{esti_local_time_4}
	%\sup_{\theta \in [0,1], a \in \real} \e[|L_T^a(V^{(n)}(\theta))|]
	%\leq \sup_{\theta \in [0,1], a \in \real} \e[|L_T^a(V^{(n)}(\theta))|^{2}]^{1/2}
	%\leq 2 \overline{\sigma} \sqrt{c(2)T}.
	%\end{align}
\end{proof}

The following lemma, which is similar to Lemma 2.2 in \cite{Y}, plays a crucial role in our argument.
%tightness of $(X^{(n)})_{n\in\n}$, Lemma 2.2 \cite{Yan}

\begin{Lem}\label{tight_1}
	Assume that $b$ and $\sigma$ are bounded measurable.
	For any $\varepsilon, \chi>0$ such that $\delta:=\frac{\chi \varepsilon^4}{8(T\|b\|_{\infty}^4+2^7\overline{\sigma}^4)}\leq T$, it holds that for any $t \geq 0$ and $n \in \n$,
%	\begin{align*} %\label{tight_2}
$		\p(\sup_{t \leq r \leq t+\delta}|X_r^{(n)}-X_t^{(n)}| \geq \varepsilon) \leq \delta \chi.$
%	\end{align*}
\end{Lem}

%\begin{Rem} 	The statement of Lemma \ref{tight_1} implies that the Euler-Maruyama approximation $(X^{(n)}_{n \in \n}$ is tight in $C[0,T]$, (see, Theorem 8.3 of \cite{Bi68}). 	We note that using tightness, Yan \cite{Y} if the sets of discontinuous points of $b$ and $\sigma$ are countable, then the Euler-Maruyama approximation converges weakly to the unique weak solution of the corresponding SDE. \end{Rem}

\begin{proof}%[Proof of Lemma \ref{tight_1}]
	Let $t\in [0,T]$ be fixed.
	We define $Z_s^{(n)}:=X_{t+s}^{(n)}-X_{t}^{(n)}$.
	Then using Burkholder-Davis-Gundy's inequality, it holds that for any $\delta \in [0,T]$,
	\begin{align*}
		\e\left[\sup_{0\leq s \leq \delta}|Z_s^{(n)}|^4\right]
		&\leq 8 \e\left[\sup_{0 \leq s \leq \delta} \left|\int_{t}^{t+s} b(X_{\eta_n(r)}^{(n)})dr\right|^4\right]
		+8 \e\left[\sup_{0 \leq s \leq \delta} \left|\int_{t}^{t+s} \sigma(X_{\eta_n(r)}^{(n)})dW_r\right|^4\right]\\
		&\leq 8 \delta^3 \e\left[ \int_{t}^{t+\delta} \left| b(X_{\eta_n(r)}^{(n)})\right|^4 dr\right]
		+2^{10} \delta \e\left[ \int_{t}^{t+\delta} \left|\sigma(X_{\eta_n(r)}^{(n)})\right|^4dr\right]\\
		&\leq 8 \|b\|^4_{\infty} \delta^4
		+2^{10} \overline{\sigma}^4 \delta^2
		\leq 8\left( \|b\|^4_{\infty}T^2+ 2^7 \overline{\sigma}^4 \right) \delta^2.
	\end{align*}
	Hence, for any $\varepsilon, \chi>0$ such that $\delta:=\frac{\chi \varepsilon^4}{8(T^2\|b\|_{\infty}^4+2^7\overline{\sigma}^4)} \leq T$, from Markov's inequality, we have
		\begin{align*}
		\p\left(\sup_{t \leq s \leq t+\delta}|X_s^{(n)}-X_t^{(n)}| \geq \varepsilon \right)
		&\leq \frac{1}{\varepsilon^4} \e\left[\sup_{t \leq s \leq t+\delta}|X_s^{(n)}-X_t^{(n)}|^4\right]
		= \frac{1}{\varepsilon^4} \e\left[\sup_{0 \leq s \leq \delta} |Z_s^{(n)}|^4 \right]\\
		&\leq \frac{8\left( \|b\|^4_{\infty}T^2+ 2^7 \overline{\sigma}^4 \right) \delta^2}{\varepsilon^4}
		=\delta \chi,
	\end{align*}
	which concludes the statement.
\end{proof}

%By using 
Lemma \ref{tight_1} directly implies the following result.
\begin{Lem}\label{tight_3}
	Assume that $b$ and $\sigma$ are bounded measurable.
	Let $(\gamma_n)_{n\in\n}$ be a decreasing sequence such that $\gamma_n \in (0,1]$ and $\gamma_n \downarrow 0 $ and $\gamma_n n^2 \to \infty$ as $n \to \infty$.
	Denote 
%	\begin{align*}%\label{seq}
$		\varepsilon_n
		:=\frac{\widetilde{c}}{\gamma_n^{1/4} n^{1/2}},~
		\widetilde{c}
		:=2^{3/4}T^{1/2}\{T^2\|b\|_{\infty}^4+ 2^7 \overline{\sigma}^4\}^{1/4},
%		\text{ and }
		\chi_n:=\frac{\gamma_n n}{T},
%	\end{align*}
%	and then, the sequence
%	\begin{align*}
		\delta_n:=\frac{\chi_n \varepsilon_n^4}{8(T^2\|b\|_{\infty}^4+ 2^7 \overline{\sigma}^4)}=\frac{T}{n}.$
%	\end{align*}
	For each $k=0,\ldots,n-1$, we define 
	\begin{align*}
	\Omega_{k,n,\varepsilon_n}
	:=\left\{ \omega \in \Omega \bigg| \sup_{t_k^{(n)} \leq s \leq t_{k+1}^{(n)}}|X_s^{(n)}(\omega)-X_{t_k^{(n)}}^{(n)}(\omega)| \geq \varepsilon_n \right\}.
	\end{align*}
	Then it holds that
	%\begin{align}\label{tight_4}
	$\p(\Omega_{k,n,\varepsilon_n}) \leq \delta_n \chi_n = \gamma_n$.
%	\end{align}	
\end{Lem}

Now we state the a key lemma of our demonstration. 
\begin{Lem}\label{key_sigma_0}
	Let Assumption \ref{Ass_1}-(ii) hold and the drift coefficient $b$ be bounded and measurable.
	Let $f \in H^{\beta,\kappa}$ for some $\beta \in (0,1]$. Then for any $p\geq 1$ and $0< \alpha < \frac{p\beta}{2} \wedge \frac{2\kappa}{\kappa+4}$, there exists a positive constant $C_p^*(f)= C^*(p,\alpha, \beta, \kappa,T,x_0,\|f\|_\beta, C_{\beta,\kappa}, \|b\|_\infty, \overline{\sigma}, \underline{\sigma})$ which does not depend on $n$ such that for each $n \geq 3$,
\begin{align} \label{eqnL0}
\int_{0}^{T}\e\left[\left|f(X_s^{(n)})-f(X_{\eta_n(s)}^{(n)})\right|^p \right]ds \leq \frac{C_p^*(f)}{n^\alpha \log n}.
\end{align}
\end{Lem}

\begin{proof}
	From Lemma \ref{tight_3} and the boundedness of $f$, it holds that
	\begin{align}\label{key_sigma_1}
		&\int_{0}^{T}\e\left[\left|f(X_s^{(n)})-f(X_{\eta_n(s)}^{(n)}) \right|^p\right]ds	\notag
		\\
		&=\sum_{k=0}^{n-1}\int_{t_{k}^{(n)}}^{t_{k+1}^{(n)}}\e\left[\left|f(X_s^{(n)})-f(X_{t_{k}^{(n)}}^{(n)}) \right|^p \left(\1_{\Omega_{k,n,\varepsilon_n}}+\1_{\Omega_{k,n,\varepsilon_n}^c} \right)\right]ds \notag\\
		&\leq 2^p\|f\|_{\infty}^pT  \gamma_n  %\p(\Omega_{k,n,\varepsilon_n})
		+\sum_{k=0}^{n-1}\int_{t_{k}^{(n)}}^{t_{k+1}^{(n)}}\e\left[\left|f(X_s^{(n)})-f(X_{t_{k}^{(n)}}^{(n)}) \right|^p \1_{\Omega_{k,n,\varepsilon_n}^c} \right]ds.
	\end{align}
We estimate the second term of \eqref{key_sigma_1} as follows
\begin{align}
&\sum_{k=0}^{n-1}\int_{t_{k}^{(n)}}^{t_{k+1}^{(n)}}\e\left[\left|f(X_s^{(n)})-f(X_{t_{k}^{(n)}}^{(n)}) \right|^p \1_{\Omega_{k,n,\varepsilon_n}^c} \right]ds \notag\\
=& \sum_{k=0}^{n-1}\int_{t_{k}^{(n)}}^{t_{k+1}^{(n)}}\e\left[\left|f(X_s^{(n)})-f(X_{t_{k}^{(n)}}^{(n)}) \right|^p \1_{\Omega_{k,n,\varepsilon_n}^c} \1_{X^{(n)}_s\in  S^{\varepsilon_n}(f)}\right]ds \notag  \\
& \qquad + \sum_{k=0}^{n-1}\int_{t_{k}^{(n)}}^{t_{k+1}^{(n)}}\e\left[\left|f(X_s^{(n)})-f(X_{t_{k}^{(n)}}^{(n)}) \right|^p \1_{\Omega_{k,n,\varepsilon_n}^c} \1_{X^{(n)}_s \not \in S^{\varepsilon_n}(f)}\right]ds. \label{eqnL1} 
\end{align}
On the set $\Omega_{k,n,\varepsilon_n}^c \cap \big\{ X^{(n)}_s \not \in S^{\varepsilon_n}(f)\big\}$, it holds that $S(f) \cap [X^{(n)}_s \wedge X^{(n)}_{t^{(n)}_k}, X^{(n)}_s \vee X^{(n)}_{t^{(n)}_k}] =\emptyset$, thus, 
$$
\left|f(X_s^{(n)})-f(X_{t_{k}^{(n)}}^{(n)}) \right|^p \1_{\Omega_{k,n,\varepsilon_n}^c} \1_{X^{(n)}_s \not \in S^{\varepsilon_n}(f)}
\leq \|f\|_\beta^p \left|X_s^{(n)}- X_{t_{k}^{(n)}}^{(n)} \right|^{p\beta}.
$$ 
This implies the second term of \eqref{eqnL1} is bounded by 
\begin{align} \label{eqnL2}
\|f\|_\beta^p \sum_{k=0}^{n-1}\int_{t_{k}^{(n)}}^{t_{k+1}^{(n)}}\e\left[ \left|X_s^{(n)}- X_{t_{k}^{(n)}}^{(n)} \right|^{p\beta} \right]ds
\leq \|f\|_\beta^p T C_{p\beta}n^{-p\beta/2},
\end{align}
where the last inequality follows from Lemma \ref{Lem_1}. 
For each constant $K_n \geq 1\vee (|x_0|+ T\|b\|_\infty)$,
the first term of \eqref{eqnL1} is bounded by
\begin{align}
&2^p\|f\|_\infty^p \sum_{k=0}^{n-1}\int_{t_{k}^{(n)}}^{t_{k+1}^{(n)}} \Big( \e\left[ \1_{X^{(n)}_s\in  S^{\varepsilon_n}(f)\cap[-K_n,K_n]}\right]+ \e\left[  \1_{X^{(n)}_s\in  S^{\varepsilon_n}(f)\backslash[-K_n,K_n]}\right]\Big)ds \notag \\
\leq & 2^p\|f\|_\infty^p \int_0^T \e\left[ \1_{X^{(n)}_s\in  S^{\varepsilon_n}(f)\cap[-K_n,K_n]}\right]ds + 2^p\|f\|_{\infty}^p\int_0^T \e\left[  \1_{|X^{(n)}_s| \geq K_n}\right]ds. \label{eqnL3}
\end{align}
Since $\sigma$ is uniformly elliptic, $ \langle X^{(n)}\rangle_t \geq \underline{\sigma}^2 t$, we obtain 
\begin{align*}
\int_0^T \e\left[ \1_{X^{(n)}_s\in  S^{\varepsilon_n}(f)\cap[-K_n,K_n]}\right]ds &\leq \underline{\sigma}^{-2}  \e\left[ \int_0^T  \1_{X^{(n)}_s\in  S^{\varepsilon_n}(f)\cap[-K_n,K_n]} d\langle X^{(n)}\rangle_s\right] \\
& = \underline{\sigma}^{-2} \e\left[ \int_\real  \1_{ S^{\varepsilon_n}(f)\cap[-K_n,K_n]}(x)L_T^{x}(X^{(n)})dx \right],
\end{align*}
where the last equation follows from the occupation time formula.
Moreover, it follows from Lemma \ref{local_time} that 
\begin{align*}
\e\left[ \int_\real  \1_{S^{\varepsilon_n}(f)\cap[-K_n,K_n]}(x) L_T^{x}(X^{(n)})dx \right] &\leq \int_\real \1_{S^{\varepsilon_n}(f)\cap[-K_n,K_n]}(x) \e[L_T^{x}(X^{(n)})]dx \\
&\leq \sup_{x \in \real} \e[L_T^{x}(X^{(n)})] \lambda\Big(S^{\varepsilon_n}(f)\cap[-K_n,K_n]\Big)\\
&\leq \{12\|b\|_{\infty}^2T^2+ 6 \overline{\sigma}^2 T\}^{1/2} C_{\beta,\kappa} K_n\varepsilon_n^\kappa.
\end{align*}
Now we consider the second term of \eqref{eqnL3}. For each $s \in [0,T]$,
\begin{align*}
\e\left[  \1_{|X^{(n)}_s| \geq K_n}\right] 
&\leq  \p \Big( \Big| \int_0^s \sigma(X^{(n)}_{\eta_n(u)}) dW_u\Big| \geq K_n - \Big| x_0+ \int_0^s b(X^{(n)}_{\eta_n(u)}) du\Big|\Big)\\
&\leq \p \Big( \Big| \int_0^s \sigma(X^{(n)}_{\eta_n(u)}) dW_u\Big| \geq K_n - \|b\|_\infty T - |x_0|\Big).
\end{align*}
Since $\langle \int_{0}^{\cdot} \sigma(X_{\eta_n(s)}^{(n)})dW_s \rangle_t \leq \overline{\sigma}^2 T$ almost surely, from Proposition 6.8 of \cite{Shigekawa} and the inequality $(a-b)^2 \geq a^2/2-b^2$ for any $a,b \in \real$, we have
	\begin{align*}
		&\p\left(\sup_{0 \leq t \leq T} \left| \int_{0}^{t} \sigma(X_{\eta_n(s)}^{(n)})dW_s \right|\geq K_n -\|b\|_{\infty}T - |x_0|\right) \notag\\
		&\leq 2\exp \left(-\frac{(K_n-|x_0|-\|b\|_{\infty}T)^2}{2\overline{\sigma}^2T}\right)
		\leq 2\exp \left(\frac{(|x_0|+\|b\|_{\infty}T)^2}{2\overline{\sigma}^2T}\right) \exp\left(-\frac{K_n^2}{4\overline{\sigma}^2T}\right).
	\end{align*}
This implies
\begin{align} \label{key_sigma_13}
\int_0^T \e\left[  \1_{|X^{(n)}_s| \geq K_n}\right]ds \leq 2T\exp \left(\frac{(|x_0|+\|b\|_{\infty}T)^2}{2\overline{\sigma}^2T}\right) \exp\left(-\frac{K_n^2}{4\overline{\sigma}^2T}\right).
\end{align}	
Gathering together the estimates \eqref{key_sigma_1} --\eqref{key_sigma_13}, we get
\begin{align}
\int_{0}^{T}\e\left[\left|f(X_s^{(n)})-f(X_{\eta_n(s)}^{(n)})\right|^p \right]ds 
\leq  &  2^p\|f\|_{\infty}^pT  \gamma_n + \|f\|_\beta^p  TC_{p\beta}n^{-p\beta/2} \notag\\
& + 2^p \|f\|_{\infty}^p \underline{\sigma}^{-2} \{12\|b\|_{\infty}^2T^2+6 \overline{\sigma}^2 T\}^{1/2} C_{\beta,\kappa} K_n\varepsilon_n^\kappa \notag\\
& + 2^{p+1} \|f\|_{\infty}^p T\exp \left(\frac{(|x_0|+\|b\|_{\infty}T)^2}{2\overline{\sigma}^2T}\right) \exp\left(-\frac{K_n^2}{4\overline{\sigma}^2T}\right). \label{eqnL6}
\end{align}
For each $0< \alpha < \frac{p\beta}{2} \wedge \frac{2\kappa}{\kappa+4}$, by choosing  
$K_n = (1+|x_0|+T\|b\|_\infty + 2\overline{\sigma}\sqrt{T\alpha}) \sqrt{\log n}$ and $\gamma_n = \frac{1}{n^\alpha \log n}$, we obtain \eqref{eqnL0} from \eqref{eqnL6}. 
\end{proof}

\subsection{Method of removal of  drift}
The following removal of drift transformation plays a crucial role in our argument. 
%In this section, we recall the drift removal method (see, \cite[Chapter 5.5.5-B]{KS}).
Suppose that $b \in L^1(\real)$. %Since $\sigma^2$ is uniformly elliptic, we  define the scale function
%\begin{align*}
%f(x) : = -2 \int_0^x \frac{b(y)}{\sigma^2(y)} dy \text{ and } 
The function $\varphi (x) := \int_0^x \exp\Big(-2 \int_0^y \frac{b(z)}{\sigma^2(z)} dz \Big) dy$ is well-defined since $\sigma^2$ is uniformly elliptic.
%\end{align*}
 %Moreover, since $\varphi''=-\frac{2 b \varphi'}{\sigma^2}$, $\varphi$ satisfies the following PDE
%\begin{align*}%\label{PDE_1}
%b(x)\varphi'(x) + \frac{1}{2} \sigma^2(x) \varphi''(x) = 0.
%\end{align*}
Define $Y_t:=\varphi(X_t)$ and $Y_t^{(n)}:=\varphi(X_t^{(n)})$.
Then by It\^o's formula we have
\begin{align*}
	Y_t
	= \varphi(x_0) + \int_0^t \varphi'(X_s) \sigma(X_s)dW_s,
	%= y_0 + \int_0^t \sigma_Y(Y_s) dW_s,
\end{align*}
and
\begin{align*}
	Y_t^{(n)}
	= \varphi(x_0)
	+\int_0^t \left( \varphi'(X_s^{(n)}) b(X_{\eta_n(s)}^{(n)})+\frac{1}{2}\varphi''(X_s^{(n)}) \sigma^2(X_{\eta_n(s)}^{(n)}) \right) ds
	+ \int_0^t \varphi'(X_s^{(n)}) \sigma(X_{\eta_n(s)}^{(n)})dW_s.
\end{align*}
%where $y_0= \varphi(x_0)$ and the function $\sigma_Y$ is given by
%\begin{align*}
%	\sigma_Y(y) := \varphi' \sigma \circ \varphi^{-1} (y).
%\end{align*}

%The following estimates shall be used repeatedly in our argument.
%We will usually use without mentioning the following elementary estimates.
To simplify the notation, we denote $K_\sigma = \overline{\sigma} \vee \underline{\sigma}^{-1}$ and $C_0 = e^{2\ksm^2 \|b\|_{L^1(\real)}}$. We will make repeated use of the following elementary lemma.
\begin{Lem}(\cite{NT_2015_2}) \label{PDE_2}
	Suppose that $b \in L^1(\real)$ and Assumption \ref{Ass_1}-(ii) holds. 
	
	\begin{itemize} \item[(i)] For any $x \in \real$,
%	\begin{align*}%\label{PDE_3}
$	C_0^{-1} \leq \varphi'(x)=\exp\Big(-2 \int_0^x \frac{b(z)}{\sigma^2(z)} dz \Big) \leq C_0.$
%	\end{align*}
	\item[(ii)] For any $x \in \real$, 
	$|\varphi''(x)| \leq 2\ksm^2 \|b\|_{\infty} \|\varphi'\|_{\infty} \leq 2\|b\|_\infty \ksm^2 C_0.$
	\item[(iii)] For any $z,w \in Dom(\varphi^{-1})$,
	\begin{align}\label{PDE_4}
	|\varphi^{-1}(z)-\varphi^{-1}(w)|
	\leq C_0 |z-w|.
	\end{align}
	\end{itemize}
	%Then $\varphi'$ is bounded and uniformly positive and the inverse function $\varphi^{-1}$ is Lipschitz continuous.	 
\end{Lem}
%The proof of Lemma \ref{PDE_2} is trivial and therefore will be omitted.

\subsection{Yamada and Watanabe approximation technique}
Under the Assumption 2.2, by using the Yamada-Watanabe approximation technique, Le Gall \cite{LeGall} show that the pathwise uniequness holds for SDE (1). We also use this technique to prove the main result
%To deal with the H\"older continuity of the diffusion coefficient $\sigma$, we use  Yamada and Watanabe approximation technique 
(see \cite{Yamada} or \cite{Gyongy}).
For each $\delta \in (1,\infty)$ and $\varepsilon \in (0,1)$, we define a continuous function $\psi _{\delta, \varepsilon}: \real \to \real^+$ with $\text{supp}\: \psi _{\delta, \varepsilon}  \subset [\varepsilon/\delta, \varepsilon]$ such that
%\begin{align*} 
$\int_{\varepsilon/\delta}^{\varepsilon} \psi _{\delta, \varepsilon}(z) dz
= 1 \text{ and } 0 \leq \psi _{\delta, \varepsilon}(z) \leq \frac{2}{z \log \delta}, \:\:\:z > 0.$
%\end{align*}
Since $\int_{\varepsilon/\delta}^{\varepsilon} \frac{2}{z \log \delta} dz=2$, there exists such a function $\psi_{\delta, \varepsilon}$.
We define a function $\phi_{\delta, \varepsilon} \in C^2(\real;\real)$ by
%\begin{align*}
$\phi_{\delta, \varepsilon}(x):=\int_0^{|x|}\int_0^y \psi _{\delta, 
	\varepsilon}(z)dzdy.$
%\end{align*}
It is easy to verify that $\phi_{\delta, \varepsilon}$ has the following useful properties: 
\begin{align} 
&|x| \leq \varepsilon + \phi_{\delta, \varepsilon}(x), \text{ for any $x \in \real $}, \label{phi3}\\ 
%&\frac{\phi'_{\delta, \varepsilon}(x)}{x}>0, \text{ for any $x \in \real 
%\setminus \{0\}$}. \label{phi1}\\
&0 \leq |\phi'_{\delta, \varepsilon}(x)| \leq 1, \text{ for any $x \in \real$} \label{phi2}, \\
\phi''_{\delta, \varepsilon}(\pm|x|)&=\psi_{\delta, \varepsilon}(|x|)
\leq \frac{2}{|x|\log \delta}{\bf 1}_{[\varepsilon/\delta, \varepsilon]}(|x|), 
\text{ for any $x \in \real \setminus\{0\}$}. \label{phi4}
\end{align}
From \eqref{PDE_4} and \eqref{phi3}, for any $t \in [0,T]$, we have
\begin{align}\label{esti_X1}
	|X_t-X_t^{(n)}|
	\leq C_0 |Y_t-Y_t^{(n)}|
	\leq C_0 \left( \varepsilon + \phi_{\delta,\varepsilon}(Y_t-Y_t^{(n)}) \right).
\end{align}
Using It\^o's formula, we have
\begin{align}\label{esti_X2}
	\phi_{\delta,\varepsilon}(Y_t-Y_t^{(n)})
	=M_t^{n,\delta,\varepsilon}
	+I_t^{(n)}
	+J_t^{(n)},
\end{align}
where
\begin{align*}
M_t^{n,\delta,\varepsilon}
&:=\int_0^t \phi'_{\delta,\varepsilon}(Y_s-Y_s^{(n)}) \left\{ \varphi'(X_s)\sigma(X_s) - \varphi'(X_s^{(n)}) \sigma(X_{\eta_n(s)}^{(n)}) \right\}dW_s,\\
I_t^{(n)}
&:=-\int_0^t \phi'_{\delta,\varepsilon}(Y_s-Y_s^{(n)})
\left\{ \varphi'(X_s^{(n)})b(X_{\eta_n(s)}^{(n)}) +\frac{1}{2} \varphi''(X_s^{(n)}) \sigma^2(X_{\eta_n(s)}^{(n)}) \right\} ds,\\
J_t^{(n)}
&:=\frac{1}{2}\int_0^t \phi''_{\delta,\varepsilon}(Y_s-Y_s^{(n)})
\left| \varphi'(X_s) \sigma(X_s) - \varphi'(X_s^{(n)}) \sigma(X_{\eta_n(s)}^{(n)})  \right|^2 ds.
\end{align*}
%In the following, we will estimate $M_t^{n,\delta,\varepsilon}, I_t^{(n)}$ and $J_t^{(n)}$ under various assumptions on $b$ and $\sigma$.

\subsection{Proof of Theorem \ref{Main_1}}
We will only present the detail proof for the case that $b \in L^1(\real)$. The proof for the case  $b \not \in L^1(\real)$ is based on the  localisation technique given in \cite{NT_2015_2} and it will be omitted. 

We fix $n \geq 3$ and a constant $0< \alpha < \frac{\beta}{2} \wedge \frac{2\kappa}{\kappa+4}$.
We first consider $I_t^{(n)}$.
%Since $\phi_{\delta,\varepsilon}'$, $\varphi'$ and $b$ are bounded, and $\sigma$ is bounded, uniformly elliptic and $1/2+\alpha$-H\"older continuous, it holds that
Since $\varphi'' = - \frac{2b\varphi'}{\sigma^2}$, 
\begin{align}
	|I_t^{(n)}|
	&\leq \int_0^T \left|\phi'_{\delta,\varepsilon}(Y_t-Y_t^{(n)}) \varphi'(X_s^{(n)}) \right|
	\left|b(X_{\eta_n(s)}^{(n)}) - \frac{b(X_s^{(n)}) \sigma^2(X_{\eta_n(s)}^{(n)})}{\sigma^2(X_s^{(n)})} \right| ds. \notag
\end{align}
Thanks to Lemma \ref{PDE_2} and estimate \eqref{phi2}, we have 
\begin{align}
	|I_t^{(n)}|
	&\leq \ksm^2 C_0 \int_0^T
	\left|b(X_{\eta_n(s)}^{(n)}) \sigma^2(X_s^{(n)}) - b(X_s^{(n)}) \sigma^2(X_{\eta_n(s)}^{(n)}) \right| ds \notag\\
	&\leq \ksm^2 C_0 
	\int_0^T \left\{
	\ksm^2 \left|b(X_s^{(n)}) - b(X_{\eta_n(s)}^{(n)}) \right|
	+\|b\|_{\infty} \left|  \sigma^2(X_s^{(n)}) - \sigma^2(X_{\eta_n(s)}^{(n)}) \right|
	\right\}ds. \notag
\end{align}
It follows from Lemma  \ref{key_sigma_0} that
\begin{equation} \label{eqnL7}
\e[|I_t^{(n)}|] \leq \frac{C_I}{n^\alpha \log n},
\end{equation}
where $C_I:=K_{\sigma}^2 C_0\{K_{\sigma}^2 C_1^*(b)+2\|b\|_{\infty} \overline{\sigma} C_1^*(\sigma) \}$.
Now we estimate $J_t^{(n)}$.
From \eqref{phi4}, we have
\begin{align*}
J_t^{(n)}
&\leq \int_0^T \frac{\1_{[\varepsilon/\delta,\varepsilon]}(|Y_s-Y_s^{(n)}|)}{|Y_s-Y_s^{(n)}| \log \delta}
\left| \varphi'(X_s) \sigma(X_s) - \varphi'(X_s^{(n)}) \sigma(X_{\eta_n(s)}^{(n)})  \right|^2 ds\\
&\leq 3(J_T^{1,n}+J_T^{2,n}+J_T^{3,n}),
\end{align*}
where
\begin{align*}
J_t^{1,n}
&:= \int_0^t \frac{\1_{[\varepsilon/\delta,\varepsilon]}(|Y_s-Y_s^{(n)}|)}{|Y_s-Y_s^{(n)}| \log \delta}
|\sigma(X_s)|^2 \left| \varphi'(X_s) - \varphi'(X_s^{(n)}) \right|^2 ds,\\
J_t^{2,n}
&:=\int_0^t \frac{\1_{[\varepsilon/\delta,\varepsilon]}(|Y_s-Y_s^{(n)}|)}{|Y_s-Y_s^{(n)}| \log \delta}
|\varphi'(X_s^{(n)})|^2 \left| \sigma(X_s) - \sigma(X_s^{(n)})  \right|^2 ds, \\
J_t^{3,n}
&:=\int_0^t \frac{\1_{[\varepsilon/\delta,\varepsilon]}(|Y_s-Y_s^{(n)}|)}{|Y_s-Y_s^{(n)}| \log \delta}
|\varphi'(X_s^{(n)})|^2 \left| \sigma(X_s^{(n)}) - \sigma(X_{\eta_n(s)}^{(n)})  \right|^2 ds.
\end{align*}
From Lemma \ref{PDE_2}-(ii), $\varphi'$ is Lipschitz continuous with Lipschitz constant $\|\varphi''\|_{\infty}$.
Hence, we have
\begin{align}\label{esti_J1}
	J_T^{1,n}
	&\leq \frac{\ksm^2 \|\varphi''\|_{\infty}^2}{\log \delta} \int_0^T \frac{\1_{[\varepsilon/\delta,\varepsilon]}(|Y_s-Y_s^{(n)}|)}{|Y_s-Y_s^{(n)}|}
	\left| X_s - X_s^{(n)} \right|^2 ds \notag\\
	&\leq \frac{\ksm^2 \|\varphi''\|_{\infty}^2 C_0^2}{\log \delta}
	\int_0^T\1_{[\varepsilon/\delta,\varepsilon]}(|Y_s-Y_s^{(n)}|) \left| Y_s - Y_s^{(n)} \right| ds \notag\\
	&\leq \frac{C_{J,1} \varepsilon}{\log \delta},
\end{align}
where $C_{J,1}:=4\ksm^6 C_0^4 \|b\|_\infty^2 T$.
Next we consider $J_T^{2,n}$. We first note that by \eqref{PDE_4},
$$J_T^{2,n} \leq \frac{C^3_0}{\log \delta} \int_0^T \frac{\left| \sigma(X_s) - \sigma(X_s^{(n)})  \right|^2}{|X_s-X_s^{(n)}|} \1_{|X_s- X^{(n)}_s|\geq  \varepsilon/(C_0\delta)} ds.$$
Recall that by Assumption \ref{Ass_1}-(i), there exists a bounded and strictly increasing function $f_{\sigma} : \real \to \real$ such that for any $x,y \in \real$,
\begin{align*}
|\sigma(x)-\sigma(y)|^2
\leq |f_{\sigma}(x)-f_{\sigma}(y)|.
\end{align*}
We consider approximation $f_{\sigma,\ell} \in C^1(\real)$ of $f_{\sigma}$ which is also strictly increasing function and satisfies $\|f_{\sigma, \ell}\|_{\infty} \leq \|f_{\sigma}\|_{\infty}$ and $f_{\sigma,\ell} \uparrow f_{\sigma}$ as $\ell \to \infty$ on $\real$.
Then by using Fatou's lemma and the mean value theorem, we have
\begin{align}\label{pr_1_5}
J_T^{2,n}
&\leq \frac{C^3_0}{\log \delta}
\int_0^T
\frac{|f_{\sigma}(X_s)-f_{\sigma}(X_s^{(n)})|}{|X_s-X_s^{(n)}|}
\1_{|X_s- X^{(n)}_s|> \varepsilon/(C_0\delta)} ds \notag\\
&\leq \liminf_{\ell \to \infty}
\frac{C^3_0}{\log \delta}
\int_0^T
\frac{|f_{\sigma, \ell}(X_s)-f_{\sigma, \ell}(X_s^{(n)})| }{|X_s-X_s^{(n)}|} \1_{|X_s- X^{(n)}_s|> \varepsilon/(C_0\delta)} ds \notag\\
&\leq  \liminf_{\ell \to \infty}
\frac{C_0^3}{\log \delta}	
\int_0^T ds \int_0^1 d\theta f'_{\sigma, \ell}(V_s^{(n)}(\theta)),
\end{align}
where $V^{(n)}(\theta)=(V_t^{(n)}(\theta))_{0 \leq t \leq T}$ is defined in Lemma \ref{local_time}. 
Since $\sigma \geq \underline{\sigma}$, the quadratic variation of $V^{(n)}(\theta)$ satisfies
\begin{align*}
\langle V^{(n)}(\theta) \rangle_t
= \int_{0}^{t} \left\{ (1-\theta)\sigma(X_s) + \theta \sigma(X_{\eta_n(s)}^{(n)}) \right\}^2 ds
\geq \underline{\sigma}^2 t,
\end{align*} 
which implies
\begin{align*}%\label{pr_1_6}
\int_0^T ds \int_0^1 d \theta f'_{\sigma, \ell}(V_s^{(n)}(\theta))
&\leq
\underline{\sigma}^{-2} \int_0^1 d \theta \int_0^T d \langle V^{(n)}(\theta) \rangle_{s} f'_{\sigma, \ell}(V_s^{(n)}(\theta)) \notag\\
&= \underline{\sigma}^{-2} \int_{\real} dx f'_{\sigma, \ell}(x) \int_0^1 d \theta L_T^x(V^{(n)}(\theta)),
\end{align*}
where the last equality is implied from the occupation time formula. Using  Lemma \ref{local_time} and  the estimate $\|f'_{\sigma, \ell}\|_{L^1(\real)} \leq 2 \|f_{\sigma, \ell}\|_{\infty} \leq 2 \|f_{\sigma} \|_{\infty}$ we have
\begin{align*}
\e\left[
\int_0^T ds \int_0^1 d \theta f'_{\sigma, \ell}(V_s^{(n)}(\theta))
\right]
&\leq \underline{\sigma}^{-2} \int_{\real} dx f'_{\sigma, \ell}(x) \int_0^1 d \theta
\e [L_T^x(V^{(n)}(\theta))] \\
&\leq \underline{\sigma}^{-2} \|f'_{\sigma, \ell}\|_{L^{1}(\real)}
\sup_{\theta \in [0,1], x \in \real}
\e [|L_T^x(V^{(n)}(\theta))|^2]^{1/2}\\
&\leq 2 \underline{\sigma}^{-2} \|f_{\sigma}\|_{\infty} 
\{ 12\|b\|_{\infty}^2T^2+6 \overline{\sigma}^2 T\}^{1/2}.
\end{align*}
By plugging this estimate to \eqref{pr_1_5} and using Fatou's lemma, we get the following estimate for the expectation of $J^{2,n}_T$,
\begin{align}\label{pr_1_7}
\e[J^{2,n}_T]
&\leq \frac{C_{J,2}}{\log \delta},
\end{align}
where $C_{J,2}:=2C_0^3\underline{\sigma}^{-2} \|f_{\sigma}\|_{\infty} \{ 12\|b\|_{\infty}^2T^2+6 \overline{\sigma}^2 T\}^{1/2}$.
Finally, we estimate $J^{3,n}_T$ as follows
\begin{align*}
\e[J^{3,n}_T] \leq \frac{C_0^2 \delta}{\varepsilon \log \delta} \int_0^T \e\Big[\left|\sigma(X^{n}_s)- \sigma(X^{(n)}_{\eta_n(s)})\right|^2\Big]ds. 
\end{align*}
Applying Lemma \ref{key_sigma_0}, we get 
\begin{equation} \label{eqnL8}
\e[J^{3,n}_t] \leq \frac{\delta}{\varepsilon \log \delta} \frac{C_{J,3}}{n^\alpha \log n},
\end{equation}
where $C_{J,3}:=C_0^2 C_2^*(\sigma)$.
Since $\e [M_t^{n,\delta, \varepsilon}] = 0$, it follows from \eqref{esti_X1} -- \eqref{eqnL8} that  there exists a positive constant $C$ which do not depend on $n$ such that
\begin{align*}
\sup_{0\leq t \leq T}\e[|X_t - X^{(n)}_t|] \leq C \Big( \varepsilon + \frac{1}{n^\alpha \log n} + \frac{\varepsilon}{\log \delta} + \frac{1}{\log \delta} + \frac{\delta}{\varepsilon \log \delta} \frac{1}{n^\alpha \log n}\Big).
\end{align*}
By choosing $\varepsilon = \frac{1}{\log n}$ and $\delta = n^\alpha$, we obtain the desired result.
\qed

\section*{Acknowledgements}
The authors thank  Arturo Kohatsu-Higa, Miguel Martinez and Toshio Yamada for their helpful comments. This research is funded by Vietnam National Foundation for Science and Technology Development (NAFOSTED).
The second author was supported by JSPS KAKENHI Grant Number 16J00894.

\end{document}